\documentclass[12pt]{article}

\usepackage{theorem}
\usepackage{amsmath,amssymb,latexsym}
\usepackage{mathabx}
\usepackage{enumitem}

\input{xy}
\xyoption{all}

\usepackage[all]{xy}

\theoremstyle{plain}
\theoremheaderfont{\normalfont\bfseries}
\theorembodyfont{\itshape}

\newtheorem{theorem}{Theorem}[section]

\newtheorem{lemma}[theorem]{Lemma}
\newtheorem{corollary}[theorem]{Corollary}

\theorembodyfont{\upshape}
{\theoremstyle{break}

}

\theorembodyfont{\upshape}
\newtheorem{remark}[theorem]{Remark}

\theorembodyfont{\upshape}

\theorembodyfont{\upshape} 
\newtheorem{definition}[theorem]{Definition}

\theorembodyfont{\upshape}
\newtheorem{question}[theorem]{Question}

\newtheorem{exam}[theorem]{Example}
\newtheorem{examples}[theorem]{Examples}

\theorembodyfont{\itshape}

\theorembodyfont{\upshape}

\theorembodyfont{\upshape}

\theorembodyfont{\upshape}

\theorembodyfont{\upshape}

\theorembodyfont{\upshape}

\theorembodyfont{\upshape}
\newtheorem{proof}{Proof}

\theorembodyfont{\upshape}

\newtheorem{acknowledgement}{Acknowledgement}

\newcommand{\myemail}[1]{\indent \emph{E-mail:} {\tt #1}}
\newcommand{\myaddress}[1]{\indent {\sc #1}\par}

\begin{document}

\title{Naive vs. genuine $\mathbb A^1$-connectedness}
\author{Anand Sawant}
\date{}

\maketitle

\begin{abstract}
We show that the triviality of sections of the sheaf of $\mathbb A^1$-chain connected components of a space over finitely generated separable field extensions of the base field is not sufficient to ensure the triviality of the sheaf of its $\mathbb A^1$-chain connected components, contrary to the situation with genuine $\mathbb A^1$-connected components.  As a consequence, we show that there exists an $\mathbb A^1$-connected scheme for which the Morel-Voevodsky singular construction is not $\mathbb A^1$-local.
\end{abstract}

\section{Introduction}
\label{section Introduction}

Let $k$ be a field and let $Sm/k$ denote the category of smooth, finite-type schemes over $k$.  In the 1990's, Morel and Voevodsky \cite{Morel-Voevodsky} constructed the \emph{$\mathbb A^1$-homotopy category} $\mathcal{H}(k)$ by taking a suitable localization of the category of simplicial sheaves of sets on $Sm/k$ for the Nisnevich topology.  Objects of $\mathcal H(k)$ are often called \emph{spaces}.  Analgous to algebraic topology, one then studies the $\mathbb A^1$-homotopy sheaves of a (pointed) space $(\mathcal X, x)$ - the sheaf of $\mathbb A^1$-connected components $\pi_0^{\mathbb A^1}(\mathcal X)$, which is a sheaf of sets and the higher homotopy sheaves $\pi_n^{\mathbb A^1}(\mathcal X, x)$, for $n\geq 1$, which are sheaves of groups.  We will use the notation and terminology of \cite{Morel-Voevodsky}.  Any (not necessarily smooth) scheme $X$ over $k$ can be viewed as an object of the $\mathbb A^1$-homotopy category $\mathcal{H}(k)$ (see the conventions stated at the beginning of Section \ref{section A1-connectedness}).   Recent works in $\mathbb A^1$-homotopy theory have indicated that the $\mathbb A^1$-homotopy sheaves of schemes are often related to some of their interesting classical invariants.

The simplest of objects in classical topology are the discrete topological spaces.  The analogous notion in $\mathbb A^1$-homotopy theory is that of \emph{$\mathbb A^1$-invariant} sheaves (see Section \ref{section A1-connectedness} for precise definitions).  In topology, the set of connected components of a topological space and the homotopy groups of a (pointed) topological space are discrete as topological spaces.  Analogously, one can ask if the $\mathbb A^1$-homotopy sheaves of a (pointed) space $\mathcal X$ are $\mathbb A^1$-invariant.  It has been shown by Morel \cite[Theorem 6.1, Corollary 6.2]{Morel} that the higher homotopy sheaves $\pi_n^{\mathbb A^1}(\mathcal X, x)$, for $n\geq 1$, are $\mathbb A^1$-invariant.  In fact, Morel shows much more - these higher $\mathbb A^1$-homotopy sheaves are \emph{strongly $\mathbb A^1$-invariant} in the sense of \cite[Definition 1.7]{Morel} (this shows that the higher $\mathbb A^1$-homotopy sheaves are very special; for instance, they are birational invariants of smooth, proper schemes).  However, $\mathbb A^1$-invariance of the sheaf of $\mathbb A^1$-connected components is not yet known; this has been conjectured by Morel \cite[Conjecture 1.12]{Morel}.  It is worthwhile to mention that $\pi_0^{\mathbb A^1}$ fails to be a birational invariant of smooth, proper schemes \cite[Example 4.8]{Balwe-Hogadi-Sawant}.  

There are two notions of $\mathbb A^1$-connectedness in unstable $\mathbb A^1$-homotopy theory.  The \emph{naive} notion is that of \emph{$\mathbb A^1$-chain connected components} of a space (see Definition \ref{definition A1-chain connected components}), which is obtained by taking the Nisnevich sheafification of the presheaf that associates with any smooth scheme $U$ the set of morphisms from $U$ to the space in question modulo the equivalence relation generated by naive $\mathbb A^1$-homotopies.  On the other hand, the \emph{genuine} notion is that of \emph{$\mathbb A^1$-connected components} (see Definition \ref{definition A1-connected components}) introduced by Morel-Voevodsky.  These two notions do not coincide in general, not even for smooth projective varieties over $\mathbb C$ (see \cite[Section 4]{Balwe-Hogadi-Sawant} for the first examples).  Given a scheme $X$ over $k$, one can infinitely iterate the construction of $\mathbb A^1$-chain connected components to obtain the so-called \emph{universal $\mathbb A^1$-invariant quotient} of $X$, which is isomorphic to $\pi_0^{\mathbb A^1}(X)$ provided the latter sheaf is $\mathbb A^1$-invariant (that is, Morel's conjecture holds for $X$).  We recall these notions and known results about them in Section \ref{section A1-connectedness}.

A natural question is to characterize genuine $\mathbb A^1$-connectedness of a scheme $X$ over $k$ in terms of triviality of sections of $\pi_0^{\mathbb A^1}(X)$ over field extensions of $k$.  A result of Morel states that $\mathbb A^1$-connectedness of a scheme $X$ over an infinite field $k$ in the genuine sense (that is, triviality of $\pi_0^{\mathbb A^1}(X)$ as a sheaf) is equivalent to the triviality of $\pi_0^{\mathbb A^1}(X)({\rm Spec}~ F)$, where $F$ runs over all finitely generated separable field extensions of $k$.  In this short note, we examine the analogous property for the sheaf of $\mathbb A^1$-chain connected components in Section \ref{section main result} (see Theorem \ref{theorem chain connected}).  As a consequence, we obtain an example of an $\mathbb A^1$-connected singular proper scheme $X$ for which the Morel-Voevodsky singular construction ${\rm Sing_*^{\mathbb A^1}} X$ is not $\mathbb A^1$-local (see Example \ref{example main theorem}).

\section{Connectedness in unstable $\mathbb A^1$-homotopy theory}
\label{section A1-connectedness}

We begin this section by setting up the notation and conventions that will be used throughout the paper.

Fix a base field $k$.  We will henceforth denote by $Sm/k$ the big Nisnevich site of smooth, finite-type schemes over $k$.  We begin with the category of simplicial sheaves over $Sm/k$.  Any scheme $X$ over $k$ can be seen as an object of this category as follows: consider the \emph{functor of points} $h_X$ of $X$, which is the sheaf that associates with every $U \in Sm/k$ the set of morphisms of schemes over $k$ from $U$ to $X$.  Any Nisnevich sheaf $\mathcal F$ on $Sm/k$ can be viewed as a simplicially constant simplicial sheaf.  More precisely, one considers the simplicial sheaf in which the sheaf at every level is $\mathcal F$ and all the face and degeneracy maps are given by the identity map.  We will always denote the simplicial sheaf corresponding to $h_X$ for a scheme $X$ over $k$ by the same letter $X$.

A morphism $\mathcal X \to \mathcal Y$ of simplicial sheaves of sets on $Sm/k$ is a \emph{local weak equivalence} if it induces an isomorphism on every stalk.  The \emph{Nisnevich local injective model structure} on this category is the one in which the morphism of simplicial sheaves is a cofibration (resp. a weak equivalence) if and only if it is a monomorphism (resp. a local weak equivalence).  The corresponding homotopy category is called the \emph{simplicial homotopy category} and is denoted by $\mathcal H_s(k)$. The left Bousfield localization of the Nisnevich local injective model structure with respect to the collection of all projection morphisms $\mathcal X \times \mathbb A^1 \to \mathcal X$, as $\mathcal X$ runs over all simplicial sheaves, is called the \emph{$\mathbb A^1$-model structure}. The associated homotopy category is called the \emph{$\mathbb A^1$-homotopy category} and is denoted by $\mathcal H(k)$.  We will denote by $\ast$ the trivial one-point sheaf on $Sm/k$.  We will abuse the notation and use $\ast$ to also denote a set with one element, whenever there is no confusion.

\begin{definition}
\label{definition A1-local}
A space (that is, a simplicial Nisnevich sheaf of sets on $Sm/k$) $\mathcal X$ is said to be \emph{$\mathbb A^1$-local} if the projection map $U \times \mathbb A^1 \to U$ induces a bijection
\[
{\rm Hom}_{\mathcal H_s(k)}(U, \mathcal X) \to {\rm Hom}_{\mathcal H_s(k)}(U \times \mathbb A^1, \mathcal X).
\]
for every $U \in Sm/k$.  Note that a Nisnevich sheaf $\mathcal F$ on $Sm/k$ is $\mathbb A^1$-local if and only if it is \emph{$\mathbb A^1$-invariant}, that is, if the projection map $U \times \mathbb A^1 \to U$ induces a bijection $\mathcal F(U) \to \mathcal F(U \times \mathbb A^1)$, for every $U \in Sm/k$.  Following standard convention, we say that a scheme is \emph{$\mathbb A^1$-rigid} if it is $\mathbb A^1$-local as a space.
\end{definition}

Let $\mathcal X$ be a space.  We now recall the \emph{singular construction} on $\mathcal X$ defined by Morel-Voevodsky \cite[p. 87-88]{Morel-Voevodsky}.  Define ${\rm Sing_*^{\mathbb A^1}} \mathcal X$ to be the simplicial sheaf given by
\[
({\rm Sing_*^{\mathbb A^1}} \mathcal X)_n = \underline{\rm Hom}(\Delta_n,\mathcal X_n), 
\]
\noindent where $\Delta_{\bullet}$ denotes the cosimplicial scheme  
\[
\Delta_n  = {\rm Spec} \left(\frac{k[x_0,...,x_n]}{(\sum_ix_i=1)}\right)
\]
\noindent with natural face and degeneracy maps analogous to the ones on topological simplices.  There exists a natural transformation $Id \to {\rm Sing_*^{\mathbb A^1}}$ such that for any simplicial sheaf $\mathcal X$, the morphism $\mathcal X \to {\rm Sing_*^{\mathbb A^1}}(\mathcal X)$ is an $\mathbb A^1$-weak equivalence.  Observe that the singular construction ${\rm Sing_*^{\mathbb A^1}}$ takes naive $\mathbb A^1$-homotopies to simplicial homotopies.

Given a simplicial sheaf of sets $\mathcal X$ on $Sm/k$, we will denote by $\pi_0(\mathcal X)$ the presheaf on $Sm/k$ that associates with $U \in Sm/k$ the coequalizer of the diagram $\mathcal X_1(U) \rightrightarrows \mathcal X_0(U)$, where the maps are the face maps coming from the simplicial data of $\mathcal X$.  We will denote by $\pi_0^s(\mathcal X)$ the Nisnevich sheafification of the presheaf $\pi_0(\mathcal X)$.

\begin{definition}
\label{definition A1-chain connected components}
The sheaf of \emph{$\mathbb A^1$-chain connected components} of a space $\mathcal X$ is defined to be $$\mathcal S(\mathcal X) := \pi_0^s({\rm Sing_*^{\mathbb A^1}} \mathcal X).$$
Thus, $\mathcal S(\mathcal X)$ is the Nisnevich sheafification of the presheaf that associates with any smooth scheme $U$ the set $\mathcal X(U)/\sim$, where $\sim$ is the equivalence relation generated by the image of $\mathcal X_1 \to \mathcal X_0 \times \mathcal X_0$, where the maps are the face maps coming from the simplicial data of $\mathcal X$ (in other words, $\sim$ is the equivalence relation generated by naive $\mathbb A^1$-homotopies).
\end{definition}

\begin{definition}
\label{definition A1-connected components}
The sheaf of \emph{$\mathbb A^1$-connected components} of a space $\mathcal X$ is defined to be $$\pi_0^{\mathbb A^1}(\mathcal X) := \pi_0^s({L_{\mathbb A^1}} \mathcal X),$$ where ${L_{\mathbb A^1}}$ denotes an $\mathbb A^1$-fibrant replacement functor.  A space $\mathcal X$ is said to be \emph{$\mathbb A^1$-connected} if $\pi_0^{\mathbb A^1}(\mathcal X) \simeq \ast$.
\end{definition}

Morel-Voevodsky explicitly describe an $\mathbb A^1$-fibrant replacement functor as follows:
\[
L_{\mathbb A^1} = Ex \circ (Ex \circ {\rm Sing}_*^{\mathbb A^1})^{\mathbb N} \circ Ex,
\]
\noindent where $Ex$ denotes a simplicial fibrant replacement functor on the model category of simplicial Nisnevich sheaves of sets over $Sm/k$ \cite[\textsection 2, Lemma 2.6, p. 107]{Morel-Voevodsky}.  There exists a natural transformation $Id \to L_{\mathbb A^1}$ which factors through the natural transformation $Id \to {\rm Sing}_*^{\mathbb A^1}$ mentioned above. For any object $\mathcal X$, the morphism $\mathcal X \to L_{\mathbb A^1}(\mathcal X)$ is an $\mathbb A^1$-weak equivalence.  A result of Morel-Voevodsky \cite[\textsection 2, Corollary 3.22]{Morel-Voevodsky} describes what happens to the natural map ${\rm Sing}_*^{\mathbb A^1} \mathcal X \to L_{\mathbb A^1}(\mathcal X)$ after applying $\pi_0^s$; we record it below for the sake of convenience.

\begin{lemma}
\label{lemma pi0 surjection}
The canonical map $\mathcal S(\mathcal X) \to \pi_0^{\mathbb A^1}(\mathcal X)$ is an epimorphism, for every space $\mathcal X$.  If ${\rm Sing}_*^{\mathbb A^1} \mathcal X$ is $\mathbb A^1$-local, then the map $\mathcal S(\mathcal X) \to \pi_0^{\mathbb A^1}(\mathcal X)$ is an isomorphism.
\end{lemma}

We will henceforth focus on a specific class of spaces, namely, sheaves of sets on the big Nisnevich site on $Sm/k$.  We will eventually specialize to the case of schemes.  Let $\mathcal F$ be a Nisnevich sheaf on $Sm/k$.  By Lemma \ref{lemma pi0 surjection}, we have a sequence of epimorphisms
\[
\mathcal F \to \mathcal S(\mathcal F) \to \mathcal S^2(\mathcal F) \to \cdots \to \mathcal S^n(\mathcal F) \to \cdots,
\]
where $\mathcal S^{n+1}(\mathcal F)$ is defined inductively to be $\mathcal S(\mathcal S^{n}(\mathcal F))$, for every $n \in \mathbb N$.  We define
\begin{equation}
L(\mathcal F): = \varinjlim_n ~ \mathcal S^n (\mathcal F).
\end{equation}

The following result was proved in \cite{Balwe-Hogadi-Sawant} (see \cite[Theorem 2.13, Remark 2.15, Corollary 2.18]{Balwe-Hogadi-Sawant}), which shows that $L(\mathcal F)$ is the \emph{universal $\mathbb A^1$-invariant quotient} of $\mathcal F$.

\begin{theorem}
\label{theorem advances}
Let $\mathcal F$ be a sheaf of sets on $Sm/k$.  Then the sheaf $L(\mathcal F)$ is $\mathbb A^1$-invariant.  Moreover, if $\mathcal G$ is an $\mathbb A^1$-invariant sheaf, then any map $\mathcal F \to \mathcal G$ factors uniquely through the epimorphism $\mathcal F \to L(\mathcal F)$.  Moreover, if $\pi_0^{\mathbb A^1}(\mathcal F)$ is $\mathbb A^1$-invariant, then the canonical map $L(\mathcal F) \to \pi_0^{\mathbb A^1}(\mathcal F)$ is an isomorphism.
\end{theorem}

We will henceforth focus on schemes over a field.  In view of Theorem \ref{theorem advances}, it is clear that a good understanding of $L(X)$ is tantamount to understanding $\pi_0^{\mathbb A^1}(X)$.  It is natural to ask the following question.

\begin{question}
\label{question stabilization}
Let $X$ be a smooth scheme over $k$.  Does there exist $n \in \mathbb N$ such that $L(X) \simeq \mathcal S^n(X)$?
\end{question}

For every scheme $X$ over a field $k$, we have the following commutative diagram in which every morphism is an epimorphism
\[
\begin{xymatrix}{
X \ar[r] \ar[rrd] & \mathcal S(X) \ar[r] \ar[rd]& \mathcal S^2(X) \ar[r]& \cdots \ar[r] & L(X) \\
&&\pi_0^{\mathbb A^1}(X) \ar[rru]&&
}\end{xymatrix}
\]
\noindent where the existence of the map $\pi_0^{\mathbb A^1}(X) \to L(X)$ making the diagram commute is a consequence of the $\mathbb A^1$-invariance of $L(X)$ (see \cite[Lemma 2.8]{Balwe-Hogadi-Sawant}).  The morphism $\pi_0^{\mathbb A^1}(X) \to L(X)$ is an isomorphism if $\pi_0^{\mathbb A^1}(
X)$ is $\mathbb A^1$-invariant.  An affirmative answer to Question \ref{question stabilization} will give a conjectural but explicit geometric description of $\pi_0^{\mathbb A^1}(X)$.  We end this section by enlisting the known examples in which such an explicit description is available.

\begin{examples}[$\mathbb A^1$-rigid varieties]
\label{examples A1-connectedness - rigid}
For an $\mathbb A^1$-rigid variety $X$, one has isomorphisms of Nisnevich sheaves $$X \simeq \mathcal S(X) \simeq \pi_0^{\mathbb A^1}(X).$$  Examples of $\mathbb A^1$-rigid varieties include $\mathbb G_m$, algebraic tori, abelian varieties, curves of genus $\geq 1$ etc.
\end{examples}

\begin{examples}[Reductive algebraic groups]
\label{examples A1-connectedness - reductive}
For any sheaf of groups $G$, it is known that $\pi_0^{\mathbb A^1}(G)$ is $\mathbb A^1$-invariant \cite{Choudhury}.  We therefore have $\pi_0^{\mathbb A^1}(G) \simeq L(G) = \varinjlim_n \mathcal S^n(G)$.  We will now focus on the case where $G$ is a reductive algebraic group over a field.
\begin{enumerate}[label=(\alph*)]
\item {\bf Isotropic groups.} Suppose that $G$ satisfies the following \emph{isotropy condition}: every almost $k$-simple component of the derived group $G_{\rm der}$ of $G$ contains a $k$-subgroup scheme isomorphic to $\mathbb G_m$.  Under this hypothesis, Asok, Hoyois and Wendt have shown that ${\rm Sing}_*^{\mathbb A^1}G$ is $\mathbb A^1$-local \cite[Theorem 2.3.2]{Asok-Hoyois-Wendt-2}.  Therefore, for $G$ satisfying the above isotropy condition, one has $$\mathcal S(G) \simeq \pi_0^{\mathbb A^1}(G) \simeq L(G).$$  The sections of this sheaf over fields can often be described explicitly.  If $G$ is a semisimple, simply connected group over an infinite field $k$ satisfying the above isotropy hypothesis, then one has isomorphisms $W(k, G) \simeq \mathcal S(G)(k) \simeq G(k)/R$.  This is a consequence of \cite[Theorem 2.3.2]{Asok-Hoyois-Wendt-2} and a classical result \cite[Th\'eor\`eme 7.2]{Gille}.  Here $W(k,G)$ denotes the \emph{Whitehead group} of $G$ and $G(k)/R$ denotes the group of \emph{R-equivalence classes} (see \cite{Gille}, for example, for precise definitions).  

\item {\bf Anisotropic groups.} Let us assume that the base field $k$ is infinite and perfect.  Suppose now that $G$ does not satisfy the above isotropy hypothesis; that means, the derived group $G_{\rm der}$ of $G$ has at least one almost $k$-simple factor which is anisotropic.  In this case, it is known that ${\rm Sing}_*^{\mathbb A^1}G$ fails to be $\mathbb A^1$-local \cite[Theorem 4.7]{Balwe-Sawant-reductive}.  We now assume that $G$ is a semisimple anisotropic group.  Note that one has $W(k, G) = G(k)$ in this case.  A result of Borel-Tits implies in this case that $\mathcal S(G)(k) \simeq G(k)$ (see \cite[Lemma 3.7]{Balwe-Sawant-IMRN} for details).  However, one has the following result (see \cite[Theorem 4.2]{Balwe-Sawant-IMRN}, \cite[Theorem 3.6]{Balwe-Sawant-reductive}): if $G$ is a semisimple, simply connected group over an infinite perfect field $k$ which does not satisfy the above isotropy hypothesis, then one has canonical isomorphisms $$\pi_0^{\mathbb A^1}(G)(k) \simeq \mathcal S^2(G)(k) \simeq G(k)/R.$$  It is worthwhile to mention here that we do not yet know whether $\pi_0^{\mathbb A^1}(G)$ agrees with $\mathcal S^2(G)$ as a sheaf.

\item {\bf $\mathbb A^1$-connected reductive algebraic groups.}  Recall that a space $\mathcal X$ is said to be $\mathbb A^1$-connected if $\pi_0^{\mathbb A^1}(\mathcal X) \simeq \ast$.  In \cite[Theorem 5.2]{Balwe-Sawant-reductive}, $\mathbb A^1$-connected reductive algebraic groups have been characterized: a reductive algebraic group $G$ over an infinite perfect field $k$ is $\mathbb A^1$-connected if and only if $G$ is semisimple, simply connected and $R$-trivial (that is $G(F)/R$ is trivial for every finitely generated separable field extension $F$ of $k$).
\end{enumerate}
\end{examples}

\begin{examples}[Proper varieties]
\label{examples A1-connectedness - proper}
\begin{enumerate}[label=(\alph*)]
\item If $X$ is a proper variety over $k$ and if $F$ is a finitely generated field extension of $k$, then one has $\mathcal S(X)(F) \simeq \pi_0^{\mathbb A^1}(X)(F)$ (see \cite[Theorem 2.4.3]{Asok-Morel}).  One also has $\mathcal S(X)(F) \simeq \mathcal S^n(X)(F)$, for every $n$ (see \cite[Theorem 3.9, Corollary 3.10]{Balwe-Hogadi-Sawant}).

\item The case of proper schemes of dimension $\leq 1$ is very easy.  For reduced, proper (possibly singular) schemes $X$ over $k$ of dimension $\leq 1$, one always has $\mathcal S(X) \simeq \mathcal S^2(X)$ (\cite[Proposition 3.13]{Balwe-Hogadi-Sawant}).  Consequently, $\mathcal S(X) \simeq \pi_0^{\mathbb A^1}(X)$.

\item If $X$ is a proper, non-uniruled surface over $k$, then one has $\mathcal S(X) \simeq \pi_0^{\mathbb A^1}(X) \simeq \mathcal S^2(X)$ (\cite[Theorem 3.14]{Balwe-Hogadi-Sawant}). 

\item The case of smooth projective ruled surfaces is surprisingly very complicated.  If $X$ is a smooth proper rational surface, then one has $\mathcal S^2(X) = \ast$ (see Corollary \ref{corollary R-trivial}).  However, one has $\pi_0^{\mathbb A^1}(X) = \ast$ as well.  We do not yet know if the sheaf $\mathcal S(X)$ for a rational surface $X$ is trivial.

Let us now assume that the characteristic of $k$ is $0$.  The case of ruled surfaces whose minimal model is of the form $\mathbb P^1 \times C$, where $C$ is a smooth projective curve of genus $\geq 1$ (note that such a $C$ is $\mathbb A^1$-rigid) is the most complicated one.  If $E$ is a $\mathbb P^1$-bundle over $C$, then one has $\mathcal S(E) \simeq \pi_0^{\mathbb A^1}(E) \simeq C$.  If $X$ is the surface obtained by blowing up one closed point on $E$ and when $k$ is assumed to be algebraically closed, one has $\mathcal S(X) \neq \mathcal S^2(X)$.  However, in this case one has $\mathcal S^2(X) \simeq \mathcal S^3(X)$.  The details will appear in a forthcoming paper \cite{Balwe-Sawant-ruled}. 

\end{enumerate}
\end{examples}

\section{Naive $\mathbb A^1$-connectedness on field-valued points}
\label{section main result}

Let $X$ be a scheme over a field $k$.  It is often much simpler to determine sections of the sheaf $\mathcal S(X)$ on smooth schemes which are the spectrum of a finitely generated separable field extension of the base field $k$.  A result of Morel (see \cite[Lemma 6.1.3]{Morel-connectivity}) states that a space $\mathcal X$ over an infinite field $k$ is $\mathbb A^1$-connected (that is, $\pi_0^{\mathbb A^1}(\mathcal X)$ is trivial) if and only if $\pi_0^{\mathbb A^1}(\mathcal X)({\rm Spec}~F) = \ast$, for every finitely generated separable field extension $F$ of $k$.  The argument given by Morel also works when the base field is finite, thanks to Gabber's presentation lemma over finite fields proved in \cite{Hogadi-Kulkarni}.  We wish to study the analogue of this result in the context of the sheaf of $\mathbb A^1$-chain connected components.  The method used here closely follows the one employed by Morel in \cite[Section 6.1]{Morel-connectivity} and in \cite[Section 3.3]{Morel-ICTP}.

\begin{lemma}
\label{lemma chain connected}
Let $V$ be an irreducible smooth scheme over $k$ and let $W \hookrightarrow V$ be the inclusion of a dense open subscheme.  Then $\mathcal S(V/W) \simeq \ast$. 
\end{lemma}
\begin{proof}
Since we have epimorphisms $V \to V/W \to \mathcal S(V/W)$, triviality of $\mathcal S(V/W)$ follows from the following statement: any point $x \in V$ has an open neighbourhood $U$ such that $\mathcal S(U/(W \cap U))$ is trivial.  

Let $Z \hookrightarrow V$ be the closed immersion of the complement of $W$, with the reduced induced subscheme structure.  By Gabber's presentation lemma (see \cite[Theorem 3.1.1]{CTHK} for the case where $k$ is infinite and \cite[Theorem 1.1]{Hogadi-Kulkarni} for the case where $k$ is finite), $x$ admits an open neighbourhood $U$ and an \'etale morphism $\pi: U \to \mathbb A^1_{V'}$, for some open subscheme $V'$ of $\mathbb A^{d-1}$, where $d$ is the dimension of $V$ at $x$, such that $\pi$ induces a closed immersion $Z \cap U \hookrightarrow \mathbb A^1_{V'}$ satisfying $Z \cap U = \pi^{-1}(\pi(Z \cap U))$ and such that $Z \cap U \to V'$ is a finite morphism.  Therefore, we have an isomorphism of Nisnevich sheaves
\[
U/ (U-Z \cap U) \xrightarrow{\sim} \mathbb A^1_{V'}/(\mathbb A^1_{V'}-\pi(Z \cap U)).
\]
\noindent Hence, it suffices to check that $\mathcal S(\mathbb A^1_{V'}/(\mathbb A^1_{V'}-\pi(Z \cap U)))$ is trivial.  Now, since $Z \cap U \to V'$ is a finite morphism, $Z \cap U \to \mathbb P^1_{V'}$ is proper.  This closed immersion does not intersect the section at infinity $s_{\infty}: V' \to \mathbb P^1_{V'}$.  By Mayer-Vietoris excision (see \cite[\textsection 3, Lemma 1.6]{Morel-Voevodsky}), we have an isomorphism of Nisnevich sheaves
\[
\mathbb A^1_{V'}/(\mathbb A^1_{V'}-\pi(Z \cap U)) \xrightarrow{\sim} \mathbb P^1_{V'}/(\mathbb P^1_{V'}-\pi(Z \cap U)).
\]
Also observe that $\mathbb A^1_{V'} \to \mathbb P^1_{V'}/(\mathbb P^1_{V'}-\pi(Z \cap U))$ is onto and that ${\rm Sing}_*^{\mathbb A^1}(\mathbb A^1_{V'}) \simeq {\rm Sing}_*^{\mathbb A^1}(V')$ (since ${\rm Sing}_*^{\mathbb A^1}$ preserves $\mathbb A^1$-weak equivalences).  Thus, the composition
\[
V' \to \mathbb A^1_{V'} \to \pi_0^{s}({\rm Sing}_*^{\mathbb A^1}(\mathbb A^1_{V'}/(\mathbb A^1_{V'}-\pi(Z \cap U)))) \to \pi_0^{s}({\rm Sing}_*^{\mathbb A^1}(\mathbb P^1_{V'}/(\mathbb P^1_{V'}-\pi(Z \cap U)))) 
\]
\noindent is surjective for any section $V' \to \mathbb A^1_{V'}$; in particular, for the zero section.  But, in $\mathbb P^1_{V'}$, the zero section is $\mathbb A^1$-homotopic to the section at infinity $s_{\infty}$.  Since $s_{\infty}(V') \subseteq \mathbb P^1_{V'}-\pi(Z \cap U)$, it follows that
\[
V \to \pi_0^{s}({\rm Sing}_*^{\mathbb A^1}(\mathbb P^1_{V'}/(\mathbb P^1_{V'}-\pi(Z \cap U)))) = \mathcal S (\mathbb P^1_{V'}/(\mathbb P^1_{V'}-\pi(Z \cap U)))
\]
is the trivial morphism, as desired. 
\end{proof}

\begin{theorem} 
\label{theorem chain connected}
Let $k$ be a field and let $\mathcal X$ be a simplicial sheaf of sets on $Sm/k$.  Suppose that $\mathcal S(\mathcal X)({\rm Spec}~ F) = \ast$, for every finitely generated separable field extension $F$ of $k$.  Then $\mathcal S(Ex{\rm Sing}_*^{\mathbb A^1}\mathcal X) \simeq \ast$.  Consequently, $\mathcal S^2(\mathcal X) \simeq \ast$.
\end{theorem}
\begin{proof}
We need to show that for every $U \in Sm/k$, the pointed set $\mathcal S({\rm Sing}_*^{\mathbb A^1} \mathcal X)(U)$ is trivial.  It suffices to show that for every morphism $U \to \mathcal S({\rm Sing}_*^{\mathbb A^1}\mathcal X)$, there is a Nisnevich cover $\mathcal V = \coprod V_i \to U$ such that the composite $V \to U \to \mathcal S({\rm Sing}_*^{\mathbb A^1}\mathcal X)$ is trivial.

\noindent \emph{Claim:} For any irreducible, smooth $k$-scheme $V$ and a morphism $\phi: V \to \mathcal X$, the composition $V \stackrel{\phi}{\to} \mathcal X \to Ex{\rm Sing}_*^{\mathbb A^1} \mathcal X \to \mathcal S(Ex{\rm Sing}_*^{\mathbb A^1} \mathcal X)$ is trivial.

\noindent \emph{Proof of the claim:}  Let $k(V)$ denote the function field of $V$.  Since 
\[
\mathcal S(\mathcal X)({\rm Spec}~ k(V)) = \underset{W \hookrightarrow V \text{nonempty open}}{\varinjlim} \mathcal S(\mathcal X)(W)
\]
\noindent is trivial by hypothesis, there exists a dense open subset $W \hookrightarrow V$ such that the composite $W \to V \stackrel{\phi}{\to} \mathcal X$ is $\mathbb A^1$-chain homotopic to the trivial morphism.  Therefore the composite of this morphism with the morphism $\mathcal X \to {\rm Sing}_*^{\mathbb A^1} \mathcal X$ is simplicially homotopic to the trivial morphism.  Choose a simplicial fibrant replacement ${\rm Sing}_*^{\mathbb A^1} \mathcal X \to Ex {\rm Sing}_*^{\mathbb A^1} \mathcal X$.  The composite $W \hookrightarrow V \stackrel{\phi}{\to} \mathcal X \to {\rm Sing}_*^{\mathbb A^1} \mathcal X \to Ex {\rm Sing}_*^{\mathbb A^1} \mathcal X$ continues to be simplicially homotopic to the trivial map.  We denote this simplicial homotopy by $H: W \times \Delta^1 \to Ex {\rm Sing}_*^{\mathbb A^1} \mathcal X$, where $H |_{W \times \{0\}}$ is the trivial map and $H |_{W \times \{1\}}$ is induced by $\phi|_W$ (here $\Delta^1$ denotes the simplicial $1$-simplex).  Consider the acyclic cofibration $V \times \{1\} ~ \cup ~ W \times \Delta^1 \to V \times \Delta^1$.  The maps $V \times \{1\} \stackrel{\sim}{\to} V \stackrel{\phi}{\to} \mathcal X \to Ex {\rm Sing}_*^{\mathbb A^1} \mathcal X$ and $H: W \times \Delta^1 \to Ex {\rm Sing}_*^{\mathbb A^1} \mathcal X$ clearly glue to give a map $\Phi: V \times \{1\} ~ \cup ~ W \times \Delta^1 \to Ex {\rm Sing}_*^{\mathbb A^1} \mathcal X$, which fits in the following commutative diagram.
\[
\begin{xymatrix}{
V \times \{1\} ~ \cup ~ W \times \Delta^1 \ar[r]^-{\Phi} \ar[d] & Ex {\rm Sing}_*^{\mathbb A^1} \mathcal X \ar[d]\\
V \times \Delta^1 \ar[r] \ar@{-->}[ru] & \ast
}\end{xymatrix}
\]
\noindent We now use the right lifting property of the projection map $Ex {\rm Sing}_*^{\mathbb A^1} \mathcal X \to \ast$ with respect to acyclic cofibrations to see that dotted arrow in the above diagram exists.  It follows that $\phi: V \to Ex {\rm Sing}_*^{\mathbb A^1} \mathcal X$ is simplicially homotopic to a morphism $\phi': V \to Ex {\rm Sing}_*^{\mathbb A^1} \mathcal X$ whose restriction to $W$ is trivial.  Thus, we get an induced morphism (of spaces) $\bar{\phi'}: V/W \to Ex {\rm Sing}_*^{\mathbb A^1} \mathcal X$.  Now, applying Lemma \ref{lemma chain connected} and commutativity of the diagram 
\[
\begin{xymatrix}{
V \ar[d] \ar[rd] &\\
V/W \ar[r] \ar[d] & Ex {\rm Sing}_*^{\mathbb A^1} \mathcal X \ar[d]\\
\mathcal S(V/W) \ar[r]& \mathcal S(Ex {\rm Sing}_*^{\mathbb A^1} \mathcal X)
}\end{xymatrix}
\]
proves the claim.

We now complete the proof of the theorem using the claim.  Since the natural map $Ex{\rm Sing}_*^{\mathbb A^1} \mathcal X \to \mathcal S(Ex{\rm Sing}_*^{\mathbb A^1}\mathcal X)$ is an epimorphism, there is a Nisnevich covering $\coprod V_i \to U$, where $V_i$ are irreducible smooth $k$-schemes such that every composite $V_i \to U \to \mathcal S(Ex{\rm Sing}_*^{\mathbb A^1}\mathcal X)$ lifts to a morphism $V_i \to Ex{\rm Sing}_*^{\mathbb A^1} \mathcal X$.  Since $Ex{\rm Sing}_*^{\mathbb A^1} \mathcal X$ is a simplicial fibrant replacement of $\mathcal X$, each map $V_i \to Ex{\rm Sing}_*^{\mathbb A^1} \mathcal X$ is represented by a map $V_i \to {\rm Sing}_*^{\mathbb A^1} \mathcal X$ in the simplical homotopy category $\mathcal H_s(k)$.
\[
\begin{xymatrix}{
&& \mathcal X \ar[d] \\
&& {\rm Sing}_*^{\mathbb A^1} \mathcal X \ar@{->>}[d] \\
\coprod V_i \ar@{-->}[rruu] \ar[rru] \ar[r] & U \ar[r]  & \mathcal S(Ex{\rm Sing}_*^{\mathbb A^1} \mathcal X)
}\end{xymatrix}
\]
Since the sheaf at simplicial level $0$ of ${\rm Sing}_*^{\mathbb A^1} \mathcal X$ is $\mathcal X_0$, and since any map from a space of simplicial dimension $0$ to another space is determined by a map at the $0$th simplicial level, this map factors through the monomorphism $\mathcal X \to {\rm Sing}_*^{\mathbb A^1} \mathcal X$.  Thus, the theorem now follows from the claim, applied to each of the maps $V_i \to \mathcal X$.
\end{proof}

\begin{corollary}
\label{corollary R-trivial}
Let $X$ be a scheme over a field $k$ such that $\mathcal S(X)({\rm Spec}~F) = \ast$, for every finitely generated separable field extension $F$ of $k$.  Then $\mathcal S^2(X) \simeq \ast$.
\end{corollary}

The condition in Corollary \ref{corollary R-trivial} can be seen to hold when $X$ is a smooth proper rational variety over a field of characteristic $0$.  In general, it holds when $X$ is a smooth proper $R$-trivial variety over a field of characteristic $0$ (see \cite[Theorem 2.4.3, Corollary 2.4.9]{Asok-Morel}).

\begin{corollary}
\label{corollary A1-connected reductive groups}
Let $G$ be an $\mathbb A^1$-connected reductive algebraic group over an infinite perfect field $k$.  Then $\mathcal S^3(G) \simeq \ast$.
\end{corollary}
\begin{proof}
This is a straightforward consequence of Theorem \ref{theorem chain connected} and Examples \ref{examples A1-connectedness - reductive} (a), (b).
\end{proof}

\begin{remark}
Note that if $G$ in Corollary \ref{corollary A1-connected reductive groups} is such that every almost $k$-simple factor of $G$ contains a copy of $\mathbb G_m$, then one has $\mathcal S(G) = \ast$ by \cite[Theorem 2.3.2]{Asok-Hoyois-Wendt-2}.
\end{remark}

We end this note with an example of a singular, projective scheme $X$ for which $\mathcal S(X) \neq \ast$, but $\mathcal S(X)({\rm Spec}~F) = \ast$, for every finitely generated field extension $F$ of $k$.  This example was pointed out to the author by Chetan Balwe.

\begin{exam} 
\label{example main theorem}
Let $k$ be a field and let $C \hookrightarrow \mathbb P^2$ be an elliptic curve over $k$.  Let $X_1$ denote the blow up of $\mathbb P^1 \times C \hookrightarrow \mathbb P^1 \times \mathbb P^2$ at the closed point $P=((1:0), Q)$, where $Q$ is a closed point on $C$.  Let $E$ denote the exceptional divisor.  We have an obvious morphism $\pi: X_1 \to C$, which is the composition of the blow-up morphism $\phi:X_1 \to \mathbb P^1 \times C$ with the projection map $\mathbb P^1 \times C \to C$.  Let $X_2$ denote the plane $\{(0:1)\} \times \mathbb P^2$; so $X_1$ and $X_2$ intersect in $\{(0:1)\} \times C$.  Let $X: = X_1 \cup X_2$.  Since $\mathcal S(X_2) = \ast$, it is easy to see that $\mathcal S(X)({\rm Spec}~F) = \ast$, for every field extension $F$ of $k$.  By \cite[Theorem 2.4.3]{Asok-Morel}, we have $\pi_0^{\mathbb A^1}(X)({\rm Spec}~F) = \ast$, for every finitely generated field extension $F$ of $k$.  By \cite[Lemma 6.1.3]{Morel-connectivity}, we have $\pi_0^{\mathbb A^1}(X) \simeq \ast$.

Let $U$ be a smooth henselian local scheme of dimension $1$ over $k$ with closed point $u$.  Let $\gamma: U \to X_1$ be a morphism that maps $u$ on $E$ and the generic point of $U$ outside $E$.  We begin by considering naive $\mathbb A^1$-homotopies of $U$ inside $X_1$, starting at $\gamma$.  Let $h$ be a morphism $\mathbb A^1 \times U \to X_1$ such that $h|_{\{0\} \times \mathbb A^1} = \gamma$.  Since $C$ is $\mathbb A^1$-rigid, the composition $\pi \circ h: \mathbb A^1 \times U \to C$ factors through the projection $\mathbb A^1 \times U \to U$.  Thus, $h: \mathbb A^1 \times U \to X_1$ factors through $X_1 \times_C U$.  By assumption, $h$ is such that the point $Q$ on $C$ is in the image of $\pi \circ h$, that is, the image of the closed fiber $\mathbb A^1 \times \{u\}$ intersects the exceptional divisor $E$.  Hence, $h^{-1}(P)$ is a closed subscheme of $\mathbb A^1 \times U$ with support contained in $\mathbb A^1 \times \{u\}$.  Since $\phi \circ h$ can be lifted to $X_1$, we see that $h^{-1}(P)$ is a closed subscheme of $\mathbb A^1 \times U$ of codimension $1$.  Therefore, the support of $h^{-1}(P)$ must be exactly $\mathbb A^1 \times \{u\}$.  Thus, $h$ maps $\mathbb A^1 \times \{u\}$ into the exceptional divisor $E$.

Now, let $h: \mathbb A^1 \times U \to X$ be a morphism such that $h|_{\{0\} \times U} = \gamma$.  Since $\mathbb A^1 \times U$ is irreducible, this implies that $h$ factors through the inclusion of $X_1$ into $X$.  The discussion in the above paragraph shows that $h$ maps the whole closed fiber of $\mathbb A^1 \times U$ on $E$.  Hence, we cannot have $\mathcal S(X)(U) = \ast$ and consequently, $\mathcal S(X) \neq \ast$.  However, Theorem \ref{theorem chain connected} implies that $\mathcal S^2(X) = \ast$.  Since $\pi_0^{\mathbb A^1}(X) \simeq \ast$ as observed above, ${\rm Sing}_*^{\mathbb A^1} X$ cannot be $\mathbb A^1$-local in view of Lemma \ref{lemma pi0 surjection}.
\end{exam}

\begin{acknowledgement}
The author thanks Chetan Balwe for suggesting Example \ref{example main theorem} as well as for his comments on this article and Marc Hoyois for a helpful discussion during the International Colloquium on $K$-theory at TIFR.  The author also thanks the referee for a careful reading of the note and for comments that helped him improve the exposition.  
\end{acknowledgement}

\myaddress{Mathematisches Institut, Ludwig-Maximilians Universit\"at, Theresienstr. 39, 
D-80333 M\"unchen, Germany.}
\myemail{sawant@math.lmu.de}

\end{document}